\documentclass[12pt]{amsart}
\usepackage{amssymb, amsthm, amsmath}

\newtheorem{question}{Question}[section]
\newtheorem{theorem}[question]{Theorem}
\newtheorem{lemma}[question]{Lemma}
\newtheorem{corollary}[question]{Corollary}
\newtheorem{definition}[question]{Definition}

\newtheorem{proposition}[question]{Proposition}

\title{Topological games and chain conditions}
\author{Santi Spadaro}
\address{Instituto de Matematica e Estatistica (IME-USP) \\ Universidade de Sao Paulo \\ Rua do Matao, 1010 - Cidade Universitaria \\ , 05508-090 Sao Paulo - SP \\ Brazil}
\email{santidspadaro@gmail.com}



\begin{document}
\maketitle

\begin{abstract}
We apply the theory of infinite two-person games to two well-known problems in topology: Suslin's Problem and Arhangel'skii's problem on $G_\delta$ covers of compact spaces. More specifically, we prove results of which the following two are special cases: 1) every linearly ordered topological space satisfying the game-theoretic version of the countable chain condition is separable and 2) in every compact space satisfying the game-theoretic version of the weak Lindel\"of property, every cover by $G_\delta$ sets has a continuum-sized subcollection whose union is $G_\delta$-dense. 
\end{abstract}

\section{Introduction}

Infinite games have been exploited in recent years to give partial answers to various important problems in General Topology, including van Douwen's $D$-space problem (see \cite{Au}), Arhangel'skii's problem on the cardinality of Lindel\"of spaces with points $G_\delta$ (see \cite{ST}, \cite{AB}) and Bell, Ginsburg and Woods's problem on the cardinality of weakly Lindel\"of first-countable regular spaces (see \cite{BS}, \cite{AS}, \cite{ASZ}). We use them to give partial ZFC answers to Suslin's Problem and Arhangel'skii's question of whether in every compact space, a cover by $G_\delta$ sets has a continuum-sized subfamily with a $G_\delta$-dense union.

It was already known to Cantor that the real line is the unique complete dense linear order without endpoints which is separable. In the first issue of Fundamenta Mathematicae, Suslin asked whether separability could be replaced with the countable chain condition in this result. Any counterexample to this assertion came to be known as a Suslin Line. The problem turned out to be independent of the usual axioms of ZFC: under $MA_{\omega_1}$ there are no Suslin Lines, and thus the answer to Suslin's question is yes. However, Suslin Lines can be found in certain models of set theory, for example under $V=L$. Various mathematicians wondered whether there was a natural strengthening of the ccc which implied a positive answer to Suslin's Problem. For example, Knaster proved that every ordered continuum with the Knaster property is separable and Shapirovskii proved that every compact space with countable tightness and Shanin's condition is separable (see \cite{T}).

Another strengthening of the ccc was suggested by Scheepers (\cite{SchCCC}) and involves a two-person game in infinitely many moves: at inning $n$ player one chooses a maximal family of pairwise disjoint open sets $\mathcal{U}_n$ and player two picks an open set $U_n \in \mathcal{U}_n$. Player two wins if $\overline{\bigcup \{U_n: n < \omega \}}=X$. Let's call \emph{playful ccc} the property that two has a winning strategy in this game. The name is justified by the fact that if $X$ contains an uncountable (maximal) pairwise disjoint family of non-empty open sets, all one has to do to win is choose that family at every inning. So one has a winning strategy in every space which does not have the countable chain condition. Hence the playul ccc implies the usual ccc.

Daniels, Kunen and Zhou \cite{DKZ} proved that, unlike the ccc, the playful ccc is productive in ZFC. We prove that every complete dense linear order with the playful ccc is separable.

The weak Lindel\"of number of a topological space $X$ ($wL(X)$) is defined as the minimum cardinal $\kappa$ such that every open cover has a $\kappa$-sized subfamily with a dense union. A space is called weakly Lindel\"of if it has countable weak Lindel\"of number. Every Lindel\"of space is clearly weakly Lindel\"of and it is not hard to prove that every space with the countable chain condition is weakly Lindel\"of. Thus the weak Lindel\"of property is somewhere between being a covering property and a chain condition. Woods \cite{W} used the weak Lindel\"of property to characterize the $C^*$-embedded subsets of the Stone-Cech compactification of the integers under CH and Bell Ginsburgh and Woods \cite{BGW} exploited it in their elegant generalization of Arhangel'skii's theorem on the cardinality of compact first-countable spaces.

Given a topological space $X$, we indicate with $X_\delta$ the space whose underlying set is $X$ and whose topology is generated by the $G_\delta$ sets of $X$. Arhangel'skii asked (see \cite{J}) whether $wL(X_\delta) \leq 2^{\aleph_0}$ for every compact space $X$. This problem remains open.

Juh\'asz gave a partial positive answer by proving that $w(X_\delta) \leq 2^{c(X)}$ for every compact space $X$ (a related result was given in \cite{KMS} for another chain-condition type cardinal invariant known as Noetherian type). In particular, Arhangel'skii's question has a positive answer for compact ccc spaces. Here we prove another partial positive result: if $X$ is a countably compact space where player two has a winning strategy in weak Lindel\"of game of length $\omega_1$ then $wL(X_\delta) \leq 2^{\aleph_0}$. 

Let's recall some standard notation regarding games.

Given collections $\mathcal{A}$ and $\mathcal{B}$ of families subsets of a topological space $X$, we indicate with $G^\kappa_1(\mathcal{A}, \mathcal{B})$ (respectively $G^\kappa_{fin}(\mathcal{A}, \mathcal{B})$) the two-player game in $\kappa$ many innings where at inning $\alpha$ player one plays $\mathcal{A}_\alpha \in \mathcal{A}$ and player two plays $A_\alpha \in \mathcal{A}_\alpha$ (respectively, $\mathcal{F}_\alpha \in [\mathcal{A}]^{<\omega}$) and player two wins if $\{A_\alpha: \alpha < \kappa \} \in \mathcal{B}$ (respectively $\bigcup \mathcal{F}_\alpha \in \mathcal{B}$).

We indicate with $\mathcal{C}^X$ the collection of all maximal families of pairwise disjoint non-empty open sets of $X$ and with $\mathcal{O}_D^X$ the collection of all open families with dense union. Obviously $\mathcal{C}^X \subset \mathcal{O}^X_D$. When there is no danger of ambiguity we will omit $X$ in the superscript.

In our proofs we will often use elementary submodels of the structure $(H(\mu), \epsilon)$. Dow's survey \cite{D} is enough to read our paper, and we give a brief informal refresher here. Recall that $H(\mu)$ is the set of all sets whose transitive closure has cardinality smaller than $\mu$. When $\mu$ is regular, $H(\mu)$ is known to satisfy all axioms of set theory, except the power set axiom. We say, informally, that a formula is satisfied by a set $S$ if it is true when all bounded quantifiers are restricted to $S$. A set $M \subset H(\mu)$ is said to be an elementary submodel of $H(\mu)$ (and we write $M \prec H(\mu)$) if a formula with parameters in $M$ is satisfied by $H(\mu)$ if and only if it is satisfied by $M$. 

The downward Lowenheim-Skolem theorem guarantees that for every $S \subset H(\mu)$, there is an elementary submodel $M \prec H(\mu)$ such that $|M| \leq |S| \cdot \omega$. This theorem is enough in many applications, but it is often useful (especially in cardinal bounds for topological spaces) to have the following closure property. We say that $M$ is $\kappa$-closed if for every $S \subset M$ such that $|S| \leq \kappa$ we have $S \in M$. For every countable set $S \subset H(\theta)$ there is always a $\kappa$-closed elementary submodel $M \prec H(\theta)$ such that $|M|=2^{\kappa}$ and $S \subset M$.

The following theorem is also used often: let $M \prec H(\mu)$ such that $\kappa + 1 \subset M$ and $S \in M$ be such that $|S| \leq \kappa$. Then $S \subset M$.

Undefined notions can be found in \cite{E} for topology and \cite{Ku} for set theory. 

\section{Arhangel'skii's problem about $G_\delta$ covers in compact spaces}

A game-theoretic version of the weak Lindel\"of property can be obtained by considering the game $G^\kappa_1(\mathcal{O}, \mathcal{O}_D)$. At inning $\alpha < \kappa$ player one chooses an open cover $\mathcal{U}_\alpha$ and player two chooses an open set $U_\alpha \in \mathcal{U}_\alpha$. Player two wins if $\overline{\bigcup \{U_\alpha: \alpha < \kappa \}}=X$. If player two has a winning strategy in $G^\kappa(\mathcal{O}, \mathcal{O}_D)$, then $wL(X) \leq \kappa$. But there are even compact spaces where player one has a winning strategy in $G^{\omega_1}_1(\mathcal{O}, \mathcal{O}_D)$ (and hence player two can't have a winning strategy in that game, see \cite{AS}).

In \cite{ASZ} we proved that the weak Lindel\"of game is the dual of the \emph{open-picking game}. It will be convenient to exploit this duality in the proof of our partial solution to Arhangel'skii's problem.

\begin{definition}
The game $G^p_o(\kappa)$ is the two-player game in $\kappa$-many innings defined as follows: at inning $\alpha < \kappa$, player one picks a point $x_\alpha \in X$ and player two chooses an open set $U_\alpha$ such that $x_\alpha \in U_\alpha$. Player one wins if $\overline{\bigcup \{U_\alpha: \alpha < \kappa \}}=X$.
\end{definition}

\begin{lemma} \cite{ASZ} {\ \\}
\begin{enumerate}
\item \label{poplayer1} Player one has a winning strategy in the $G^\kappa_1(\mathcal{O}, \mathcal{O}_D)$ if and only if player two has a winning strategy in $G^p_o(\kappa)$.
\item \label{poplayer2} Player two has a winning strategy in $G^\kappa_1(\mathcal{O}, \mathcal{O}_D)$ if and only if player one has a winning strategy in $G^p_o(\kappa)$.
\end{enumerate}
\end{lemma}

\begin{proof}
We prove only the direct implication of ($\ref{poplayer2}$), because it's the only one we will need in our proof of Theorem $\ref{ArhanQ}$ below, and we refer the reader to \cite{ASZ} for the other implications.

Let $\sigma$ be a winning strategy for player two in $G^\kappa_1(\mathcal{O}, \mathcal{O}_D)$ on some space $X$.

\noindent {\bf Claim}. Let $(\mathcal{O}_\alpha: \alpha < \beta)$ be a sequence of open covers. Then there is a point $x \in X$ such that, for every neighbourhood $U$ of $x$ there is an open cover $\mathcal{U}$ with $U=\sigma((\mathcal{O}_\alpha: \alpha < \beta)^\frown (\mathcal{U}))$.

\begin{proof}[Proof of Claim]
Recalling that $\mathcal{O}$ denotes the set of all open covers of $X$, let $\mathcal{V}=\{V$ open: $(\forall \mathcal{U} \in \mathcal{O})(V \neq \sigma((\mathcal{O}_\alpha: \alpha < \beta)^\frown(\mathcal{U})) \}$.  Its definition easily implies that $\mathcal{V}$ cannot be an open cover, and hence there is a point $x \in X \setminus \bigcup \mathcal{V}$. By definition of $\mathcal{V}$ we must have that for every neighbourhood $U$ of $x$ there is an open cover $\mathcal{U}$ such that $U=\sigma((\mathcal{O}_\alpha: \alpha < \beta)^\frown (\mathcal{U}))$ and hence we are done.
\renewcommand{\qedsymbol}{$\triangle$}
\end{proof}

We are now going to define a winning strategy $\tau$ for player one in $G^p_o(\kappa)$.

Use the Claim to choose a point $x_0$ such that, for every neighbourhood $U$ of $x_0$ there is an open cover $\mathcal{U}$ with $\sigma((\mathcal{U}))=U$ and let $\tau(\emptyset)=x_0$.

Suppose we have defined $\tau$ for the first $\alpha$ many innings. Let now $\{V_\beta: \beta \leq \alpha \}$ be a sequence of open sets and $\{\mathcal{O}_\beta: \beta < \alpha\}$ be a sequence of open covers such that $V_\beta=\sigma((\mathcal{O}_\gamma: \gamma \leq \beta ))$, for every $\beta < \alpha$. Use the claim to choose a point $x_\alpha$ such that, for every open neighbourhood $U$ of $x_\alpha$ there is an open cover $\mathcal{O}$ with $U=\sigma((\mathcal{O}_\beta: \beta < \alpha)^\frown (\mathcal{O}))$ and let $\tau((V_\beta: \beta \leq\alpha))=x_\alpha$.

We now claim that $\tau$ is a winning strategy for player one in $G^p_o(\kappa)$. Indeed, let $(x_0, V_0, x_1, V_1, \dots x_\alpha, V_\alpha, \dots )$ be a play where player one uses strategy $\tau$. Then there must be a sequence of open covers $\{\mathcal{O}_\alpha: \alpha < \kappa \}$ such that $V_\beta=\sigma((\mathcal{O}_\alpha: \alpha < \beta ))$, for every $\beta < \kappa$. Since $\sigma$ is a winning strategy for two in $G^\kappa_1(\mathcal{O}, \mathcal{O}_D)$ then $\bigcup \{V_\alpha: \alpha < \kappa \}$ is dense in $X$ and this proves that $\tau$ is a winning strategy for player one in $G^p_o(\kappa)$.
\end{proof}

\begin{theorem} \label{ArhanQ}
Let $\kappa$ be an uncountable cardinal and $X$ be a countably compact regular space where player two has a winning strategy in $G^{\kappa}_1(\mathcal{O}, \mathcal{O}_D)$. Then $wL(X_\delta) \leq 2^{<\kappa}$.
\end{theorem}

\begin{proof}
Denote by $\rho$ the set of all open subsets of $X$. Fix a winning strategy $\tau$ for player one in $G^p_o(\kappa)$ and let $\mathcal{U}$ be an open cover of $X_\delta$. Since $X$ is regular, we can assume without loss of generality that, for every $U \in \mathcal{U}$, there are $\{U_n: n < \omega \} \subset \rho$ such that $\overline{U_{n+1}} \subset U_n$ and $U=\bigcap \{U_n: n < \omega \}=\bigcap \{\overline{U_n}: n < \omega \}$. Let $M$ be a $<\kappa$-closed elementary submodel of $H(\theta)$, for large enough regular $\theta$ such that $X, \rho, \tau, \mathcal{U} \in M$, $|M|=2^{<\kappa}$ and $2^{<\kappa}+1 \subset M$. We claim that $\mathcal{U} \cap M$ is dense in $X_\delta$. Suppose this is not the case and let $V$ be an open subset of $X_\delta$ such that $V \cap \bigcup (\mathcal{U} \cap M) =\emptyset$. We can assume that $V=\bigcap \{V_n: n <\omega \}$, where $\{V_n: n <\omega \}$ is a family of open sets of $X$ such that $\overline{V_{n+1}} \subset V_n$, for every $n<\omega$. 

\noindent {\bf Claim}. $X \cap M$ is countably compact.

\begin{proof}[Proof of Claim] Suppose that is not true and let $A \subset X \cap M$ be a countable set with no accumulation points in $X \cap M$. By countable compactness of $X$, $A$ must have an accumulation point $x \in X$. In other words we have: 
$$H(\theta) \models (\exists x \in X)(\forall U \in \rho)(x \in U \Rightarrow U \cap A \neq \emptyset).$$ 

By $<\kappa$-closedness of $M$ we have that $A \in M$, and hence, by elementarity 
$$M \models (\exists x \in X)(\forall U \in \rho)(x \in U \Rightarrow U \cap A \neq \emptyset).$$ 

This means that there is $p \in X \cap M$, such that for every neighbourhood $U$ of $p$ with $U \in M$ we have $U \cap A \neq \emptyset$. It follows that:
$$M \models (\forall U \in \rho)(p \in U \Rightarrow U \cap A \neq \emptyset).$$
Hence, by elementarity:
$$H(\theta) \models (\forall U \in \rho)(p \in U \Rightarrow U \cap A \neq \emptyset).$$ 

But that contradicts the fact that $A$ has no accumulation points in $X \cap M$.
\renewcommand{\qedsymbol}{$\triangle$}
\end{proof}

For all $x \in X \cap M$ there is $B_x \in \mathcal{U} \cap M$ such that $x \in B_x$, and hence $B_x \cap V=\emptyset$. Since $B_x \in M$, there are $\{B^x_n: n < \omega\} \subset M \cap \rho$ such that $B_x=\bigcap B^x_n$ and $\overline{B^x_{n+1}} \subset B^x_n$. Fix $x \in X \cap M$. By compactness, $B_x \cap V=\emptyset$ implies the existence of positive integers $m(x)$ and $n(x)$ such that $B^x_{m(x)} \cap V^x_{n(x)}=\emptyset$. Let $\mathcal{B}=\{B^x_{m(x)}: x \in X \cap M \}$ and $\mathcal{B}_n=\{B \in \mathcal{B}: B \cap V_n=\emptyset \}$. Set $B_n=\bigcup \mathcal{B}_n$. Then $\{B_n: n < \omega \}$ is an open cover of $X \cap M$ and hence by the Claim there is an integer $k<\omega$ such that $\{B_n: n \leq k \}$ covers $X \cap M$. Let $\mathcal{B}'=\bigcup \{\mathcal{B}_n: n \leq k \} \subset M$ and note that $\mathcal{B'}$ is a cover of $X \cap M$.

We're going to play a game of $G^p_0(\kappa)$ where player one uses $\tau$ and player two picks their moves inside $\mathcal{B}'$.

More precisely, in the first inning player one plays the point $x_0=\tau(\emptyset)$. Note that, since $\tau \in M$, $x_0 \in X \cap M$, there is an open set $B_0 \in \mathcal{B}'$ such that $x_0 \in B_0$. Let $\alpha < \kappa$ and $B_\beta \in \mathcal{B}'$ be the open set played by player two at inning $\beta$, for every $\beta < \alpha$. Since $\alpha$ is countable and $M$ is $<\kappa$-closed, we have $\{B_\beta: \beta < \alpha \} \in M$ and hence $x_\alpha=\tau ((B_\beta: \beta < \alpha)) \in M$. So there is $B_\alpha \in \mathcal{B}'$ such that $x_\alpha \in B_\alpha$. Since $\tau$ is a winning strategy for player one, we must have $\overline{\bigcup \{B_\alpha: \alpha < \kappa \}}=X$, but this contradicts the fact that $B_\alpha \cap V_k=\emptyset$, for every $\alpha < \omega_1$. 
\end{proof}

\begin{corollary}
Let $X$ be a countably compact space where player two has a winning strategy in $G^{\omega_1}_1(\mathcal{O}, \mathcal{O}_D)$. Then $wL(X_\delta) \leq 2^{\aleph_0}$.
\end{corollary}

\begin{lemma}
Let $X$ be a space such that $c(X) \leq \kappa$. Then player two has a winning strategy in $G^{\kappa^+}_1(\mathcal{O}, \mathcal{O}_D)$.
\end{lemma}

\begin{proof}
We describe the strategy by induction. Let $\beta < \kappa^+$ and suppose player two has picked the open set $U_\alpha$ at inning $\alpha$ for every $\alpha < \beta$. Suppose we have chosen open sets $\{V_\alpha: \alpha < \beta \}$ such that $\{U_\alpha \cap V_\alpha: \alpha < \beta \}$ is a pairwise disjoint family of open sets. If $\bigcup \{U_\alpha: \alpha < \beta \}$ is dense in $X$ then player two has won, otherwise let $V_\beta$ be a non-empty open set such that $V_\beta \cap (\bigcup \{U_\alpha: \alpha < \beta \})=\emptyset$. Suppose at inning $\beta$ player one plays the open cover $\mathcal{O}_\beta$. Choose an open set $U_\beta \in \mathcal{O}_\beta$ such that $U_\beta \cap V_\beta \neq \emptyset$ and let player two play $U_\beta$ at inning $\beta$. If this could be carried on for $\kappa^+$ many moves then $\{V_\alpha: \alpha < \kappa^+\}$ would be a pairwise disjoint family of non-empty open sets having size $\kappa^+$, which contradicts $c(X) \leq \kappa$. Therefore there must be $\beta < \kappa^+$ such that $\overline{\bigcup \{U_\alpha: \alpha < \beta\}}=X$.
\end{proof}

\begin{corollary}
(Juh\'asz) Let $X$ be a compact space. Then $wL(X_\delta) \leq 2^{c(X)}$.
\end{corollary}

\begin{question}
Let $X$ be a (countably) compact space such that player II has a winning strategy in $G^{\omega_1}_1(\mathcal{C}, \mathcal{O}_D)$. Is then $c(X_\delta) \leq 2^{\aleph_0}$?
\end{question}

We note that consistently, the above question has a positive answer.

\begin{proposition}
$(2^{\aleph_0}=2^{\aleph_1})$ Let $X$ be a countably compact space such that player two has a winning strategy in $G^{\omega_1}_1(\mathcal{C}, \mathcal{O}_D)$. Then $c(X_\delta) \leq 2^{\aleph_0}$.
\end{proposition}

\begin{proof}
It's easy to see that if $c(X) > \aleph_1$, then one has a winning strategy in $G^{\omega_1}_1(\mathcal{C}, \mathcal{O}_D)$. It follows that $c(X) \leq \aleph_1$. Now, Juh\'asz \cite{J} proved that $c(Y_\delta) \leq 2^{c(Y)}$ for every compact space $Y$, and hence $c(X_\delta) \leq 2^{\aleph_1}=2^{\aleph_0}$.
\end{proof}

\begin{question}
Is there a countably compact space such that player II has a winning strategy in $G^{\omega_1}_1(\mathcal{O}, \mathcal{O}_D)$ and $wL(X_\delta)>\aleph_1$?
\end{question}

\begin{question}
Let $X$ be a countably compact space such that player II has a winning strategy in $G^\kappa_{fin}(\mathcal{O}, \mathcal{O}_D)$. Is $wL(X_\delta) \leq 2^{<\kappa}$?
\end{question}

\section{The Suslin Problem for the playful ccc}

Recall that a $\pi$-base is a family $\mathcal{P}$ of non-empty open sets such that for every open set $U \subset X$ there is $P \in \mathcal{P}$ such that $P \subset U$. The $\pi$-weight of $X$ ($\pi w(X)$) is defined as the minimal size of a $\pi$-base.

A local $\pi$-base at $x \in X$ is a family $\mathcal{P}$ of non-empty open subsets of $X$ such that, for every open neighbourhood $U$ of $x$ there is $P \in \mathcal{P}$ such that $P \subset U$. The local $\pi$-character of $x$ ($\pi \chi(x, X)$) is defined as the minimum cardinality of a local $\pi$-base at $x$.

We're going to prove the following theorem.

\begin{theorem} \label{mainthm}
Let $(X, \tau)$ be a regular space with a dense set of points of countable $\pi$-character. If player II has a winning strategy in $G^\omega_1(\mathcal{C}, \mathcal{O}_D)$ then $X$ has a countable $\pi$-base.
\end{theorem}

Before going ahead to the proof, let us see how the announced partial ZFC solution to Suslin's Problem follows as a corollary.

\begin{lemma} \label{continuous}
Let $L$ be a complete dense linear order. Then $L$ contains a dense set of points of countable $\pi$-character.
\end{lemma}

\begin{proof}
Let $(a,b)$ be a non-empty open interval. We claim that $(a,b)$ contains a point of countable $\pi$-character. Let $c \in (a,b)$ and suppose you have constructed $\{x_i: i \leq n \} \subset (a,c)$. Choose $x_{n+1} \in (x_n,c)$. Then $y=\sup \{x_n: n < \omega \} \in (a,b)$ and $\{(x_n,y): n < \omega \}$ is a local $\pi$-base at $y$.
\end{proof}

\begin{corollary}
Let $X$ be a continuous linearly ordered topological space without isolated points such that player two has a winning strategy in $G^\omega_1(\mathcal{C}, \mathcal{O}_D)$. Then $X$ has a countable $\pi$-base.
\end{corollary}

Although we could give a more direct proof of Theorem $\ref{mainthm}$ we prefer to exploit the duality between the game $G^\kappa_1(\mathcal{C}, \mathcal{O}_D)$ and the open-open game. We believe this duality to have been known for some time, but since we couldn't find a proof of it in the literature, we provide one in Theorem $\ref{thmequiv}$ for the reader's convenience.

The open-open game of length $\kappa$ ($G^o_o(\kappa)$) is the two-player game where at inning $\alpha < \kappa$ player one chooses a non-empty open set $U_\alpha \subset X$ and player II chooses a non-empty open set $V_\alpha \subset U_\alpha$. Player I wins if $\overline{\bigcup \{V_\alpha: \alpha < \kappa \}}=X$ (see \cite{DKZ}).

First of all we note that the game $G^\kappa_1(\mathcal{C}, \mathcal{O}_D)$ is equivalent to the game $G^\kappa_1(\mathcal{O}_D, \mathcal{O}_D)$.

\begin{proposition}
Player I (Player II) has a winning strategy in $G^\kappa_1(\mathcal{O}_D, \mathcal{O}_D)$ if and only if Player I (Player II) has a winning strategy in $G^\kappa_1(\mathcal{C}, \mathcal{O}_D)$.
\end{proposition} 

Finally, we prove that the latter game is the dual of the open-open game. 

\begin{theorem} \label{thmequiv} {\ \\}
\begin{enumerate}
\item \label{playerone} Player one has a winning strategy in $G^o_o(\kappa)$ if and only if player two has a winning strategy in $G^\kappa_1(\mathcal{O}_D, \mathcal{O}_D)$.
\item \label{playertwo} Player two has a winning strategy in $G^o_o(\kappa)$ if and only if player one has a winning strategy in $G^\kappa_1(\mathcal{O}_D, \mathcal{O}_D)$.
\end{enumerate}
\end{theorem}

\begin{proof}
To prove the direct implication of ($\ref{playerone}$), let $\tau$ be a winning strategy for player one in $G^o_o(\kappa)$. Given $\mathcal{O} \in \mathcal{O}_D$, let $\sigma((\mathcal{O}))$ be any open set $O \in \mathcal{O}$ such that $O \cap \tau(\emptyset) \neq \emptyset$.

Now, suppose we have defined $\sigma$ for sequences of order-type $\leq \alpha$ and let $(\mathcal{O}_\beta: \beta \leq \alpha)$ be a sequence of members of $\mathcal{O}_D$. Then just let $\sigma ((\mathcal{O}_\beta: \beta \leq \alpha ))$ be any open set $O \in \mathcal{O}_\alpha$ such that $\tau((\sigma((\mathcal{O}_\beta: \beta \leq \gamma)): \gamma < \alpha )) \cap O \neq \emptyset$. We claim that $\sigma$ is a winning strategy for player two in $G^\kappa_1(\mathcal{O}_D, \mathcal{O}_D)$. Indeed, let $(\mathcal{O}_0, O_0, \mathcal{O}_1, O_1, \dots, \mathcal{O}_\alpha, O_\alpha, \dots)$ be a play of $G^\kappa_1(\mathcal{O}_D, \mathcal{O}_D)$, where player two plays according to $\sigma$. Then the set $V_\alpha=\tau((\sigma((\mathcal{O}_\gamma: \gamma \leq \beta)): \beta < \alpha)) \cap O_\alpha$ is non-empty. But since $V_\alpha \subset \tau((\sigma((\mathcal{O}_\gamma: \gamma \leq \beta)): \beta < \alpha))$ and $\tau$ is a winning strategy for player I in $G^o_o(\kappa)$ we must have that $\bigcup_{\alpha < \kappa} V_\alpha$ is dense. Hence $\bigcup_{\alpha < \kappa} O_\alpha$ is dense too and we are done.

Conversely, let $\sigma$ be a winning strategy for player II in $G^\kappa_1(\mathcal{O}_D, \mathcal{O}_D)$.

\noindent {\bf Claim.} Let $\{\mathcal{O}_\alpha: \alpha <\beta\} \subset \mathcal{O}_D$. Then there is an open set $V$ such that for every $U \subset V$ there is $\mathcal{O} \in \mathcal{O}_D$ with $U=\sigma((\mathcal{O}_\alpha: \alpha < \beta)^\frown (\mathcal{O}))$.

\begin{proof}[Proof of Claim] Let $\mathcal{U}=\{U$ open: $(\forall \mathcal{O} \in \mathcal{O}_D) (U \neq \sigma ((\mathcal{O}_\alpha: \alpha < \beta)^\frown (\mathcal{O})) \}$. By definition $\mathcal{U} \notin \mathcal{O}_D$, so there must be a non-empty open set $V$ such that $V \cap \bigcup \mathcal{U}=\emptyset$. By definition of $\mathcal{U}$, for every $U \subset V$ open there must be an $\mathcal{O} \in \mathcal{O}_D$ such that $U=\sigma ((\mathcal{O}_\alpha: \alpha < \beta)^\frown (\mathcal{O}))$, as we wanted.
\renewcommand{\qedsymbol}{$\triangle$}
\end{proof}

Now, use the claim to choose an open set $V_0$ such that for all $U \subset V_0$ there is $\mathcal{O} \in \mathcal{O}_D$ such that $\sigma((\mathcal{O}))=U$ and let $\tau(\emptyset)=V_0$.

Suppose you have defined $\tau$ for every sequence of order type $\leq \alpha$. Let $\{U_\beta: \beta \leq \alpha \}$ be a sequence of open sets and $\{\mathcal{O}_\beta: \beta \leq \alpha \}$ be a sequence of elements of $\mathcal{O}_D$ such that $U_\beta=\sigma((\mathcal{O}_\gamma: \gamma \leq \beta ))$. Use the Claim to choose an open set $V_\alpha$ such that for every open set $U \subset V_\alpha$ there is $\mathcal{O} \in \mathcal{O}_D$ such that $U=\sigma((\mathcal{O}_\beta: \beta < \alpha)^\frown (\mathcal{O}))$ and define $\tau((U_\beta: \beta \leq \alpha))$ to be $V_\alpha$. We claim that $\tau$ is a winning strategy for player one in $G^o_o(\kappa)$.

Indeed, let $(V_0, U_0, V_1, U_1, \dots V_\alpha, U_\alpha, \dots)$ be a play of $G^o_o(\kappa)$ where player one plays according to $\tau$. Then there must be a sequence $(\mathcal{O}_\alpha: \alpha < \kappa)$ of elements of $\mathcal{O}_D$ such that $U_\beta=\sigma((\mathcal{O}_\alpha: \alpha \leq \beta))$, for all $\beta < \kappa$. Since $\sigma$ is a winning strategy for player two in $G^\kappa_1(\mathcal{O}_D, \mathcal{O}_D)$ we must have that $\bigcup_{\beta < \kappa} U_\beta$ is dense and that proves that $\tau$ is a winning strategy for player one in $G^o_o(\kappa)$.

To prove the direct implication of ($\ref{playertwo}$), let $\tau$ be a winning strategy for player two in $G^o_o(\kappa)$. Call $\rho$ the set of all open sets of $X$. We first let $\sigma(\emptyset)=\{\tau(O): O \in \rho \}$. Now suppose we have defined $\sigma$ for sequences of order-type $\leq \alpha$ in such a way that if $U_\beta$ is the open set played by player two in the $\beta$-th inning, then there are open sets $\{V_\beta: \beta \leq \alpha \}$ with $U_\beta=\tau((V_\gamma: \gamma \leq \beta))$. We simply define $\sigma((U_\beta: \beta \leq \alpha))$ to be $\{\tau((V_\gamma: \gamma \leq \alpha)^\frown (O)): O \in \rho \}$. We now check that $\sigma$ is a winning strategy for player one in $G^\kappa_1(\mathcal{O}_D, \mathcal{O}_D)$. Let $\{\mathcal{O}_0, U_0, \mathcal{O}_1, U_1, \dots \mathcal{O}_\alpha, U_\alpha, \dots \}$. be a play where player one plays according to $\sigma$. So we can find a sequence of open sets $\{V_\beta: \beta < \kappa \}$ such that $U_\alpha=\tau((V_\beta: \beta \leq \alpha ))$ and hence $\bigcup \{U_\alpha: \alpha < \kappa \}$ is not dense, since $\tau$ is a winning strategy for player two in $G^o_o(\kappa)$. So $\sigma$ is a winning strategy for player one in $G^\kappa_1(\mathcal{O}_D, \mathcal{O}_D)$.

To prove the converse implication of ($\ref{playertwo}$), let $\sigma$ be a winning strategy for player one in $G^\kappa_1(\mathcal{O}_D, \mathcal{O}_D)$. Given an open set $U$, let $V \in \sigma(\emptyset)$ be any open set such that $U \cap V \neq \emptyset$ and set $\tau((U))=U \cap V$. Suppose $\tau$ has been defined for all sequences of open sets of order type $\leq \alpha$. Given $(U_\beta: \beta \leq \alpha) \subset \rho$, let $V$ be any open set $V \in \sigma((\tau((U_\gamma: \gamma \leq \beta)): \beta < \alpha))$ such that $U_\alpha \cap V \neq \emptyset$ and set $\tau((U_\beta: \beta \leq \alpha))=U_\alpha \cap V$. We claim that $\tau$ thus defined is a winning strategy for player two in $G^o_o(\kappa)$. Indeed, let $(V_0, U_0, V_1, U_1, \dots V_\alpha, U_\alpha, \dots)$ be a play where player two plays according $\tau$.  Then, for every $\alpha < \kappa$ there is $ G_\alpha \in \sigma((\tau((U_\gamma: \gamma \leq \beta)): \beta <\kappa ))$ such that $U_\alpha \subset G_\alpha$. Now $\bigcup_{\alpha < \kappa} G_\alpha$ is not dense becuse $\sigma$ is a winning strategy for player one in $G^\kappa_1(\mathcal{O}_D, \mathcal{O}_D))$ and hence $\bigcup \{U_\alpha: \alpha < \kappa \}$ can't be dense either. Therefore $\tau$ is a winning strategy for player two in $G^o_o(\kappa)$ and we are done.

\end{proof}

\begin{theorem}
Let $\kappa$ be an infinite cardinal and $X$ be a regular space with a dense set of points of $\pi$-character $\leq 2^{<\kappa}$ where player one has a winning strategy in $G^o_o(\kappa)$. Then $\pi w(X) \leq 2^{<\kappa}$.
\end{theorem}

\begin{proof}
Denote by $\rho$ the set of all open sets of $X$. Fix a winning strategy $\tau$ for player one in $G^o_o(\kappa)$ and let $D=\{x \in X: \pi \chi(x,X) \leq 2^{<\kappa} \}$. By assumption, the set $D$ is dense in $X$. Let $\theta$ be a large enough regular cardinal and $M$ be a $<\kappa$-closed elementary submodel of $H(\theta)$ such that $X, \rho, \tau, D \in M$, $|M|=2^{<\kappa}$ and $2^{<\kappa} \subset M$.

We claim that $D \cap M$ is dense in $X$. Suppose this is not the case. Then there is an open set $G \subset X$ such that $\overline{G} \cap D \cap M=\emptyset$. Note that nevertheless $D \cap M$ is dense in the possibly coarser topology generated by $\rho \cap M$. Now, the first move of player one $\tau(\emptyset)$ is an open set belonging to $M$ and thus we can fix a point $x_1 \in \tau(\emptyset) \cap D \cap M$. Let $U_1$ be an open neighbourhood of $x_1$ such that $U_1 \cap G=\emptyset$. Let $\mathcal{P}(x_1) \in M$ be a local $\pi$-base at $x_1$ having size $2^{<\kappa}$. We actually have $\mathcal{P}(x_1) \subset M$, so we can find $P_1 \in M$ such that $P_1 \subset U_1 \cap \tau(\emptyset)$. We let player two choose $P_1$ in their first move. 

Let $\beta < \kappa$ and suppose that player two has picked open sets $\{P_\alpha: \alpha < \beta \} \subset M$ such that $P_\alpha \cap G=\emptyset$. By $<\kappa$-closedness of $M$ we have $\{P_\alpha: \alpha < \beta \} \in M$, and since $\tau \in M$ we have $\tau((P_\alpha: \alpha < \beta)) \in M$. Hence we can find a point $x_\beta \in D \cap M \cap \tau((P_\alpha: \alpha < \beta))$. Let $U_\beta$ be an open neighbourhood of $x_\beta$ disjoint from $G$. Let $\mathcal{P}(x_\beta) \in M$ be a local $\pi$-base at $x_\beta$ having size $2^{<\kappa}$. We actually have $\mathcal{P}(x_\beta) \subset M$ and hence we can fix an open set $P_\beta \in M$ such that $P_\beta \subset \tau((P_\alpha: \alpha < \beta)) \cap U_\beta$. We let player two pick $P_\beta$ in the $\beta$-th inning.

Since $\tau$ is a winning strategy for player one in $G^o_o(\kappa)$ we must have that $\bigcup \{P_\alpha: \alpha < \kappa\}$ is dense in $X$, but this contradicts the fact that $P_\alpha \cap G=\emptyset$, for every $\alpha < \kappa$

Therefore $D \cap M$ is dense in $X$ and hence $X$ has a $2^{<\kappa}$-sized dense set of points of $\pi$-character $2^{<\kappa}$. Now, putting together the $2^{<\kappa}$-sized $\pi$-bases at each of the points of $D \cap M$ one gets a $2^{<\kappa}$-sized (global) $\pi$-base for our space.
\end{proof}

\begin{corollary}
Let $X$ be a regular space with a dense set of points of countable $\pi$-character where player one has a winning strategy in $G^o_o(\omega)$. Then $X$ has a countable $\pi$-base.
\end{corollary}

\begin{question}
Is there a space such that player II has a winning strategy in $G^{\omega_1}_1(\mathcal{C}, \mathcal{O}_D)$ and $\pi w(X) >\aleph_1$?
\end{question}

\begin{question}
Let $X$ be a space with a dense set of points of countable $\pi$-character were player two has a winning strategy in $G^\omega_{fin}(\mathcal{O}_D, \mathcal{O}_D)$. Is it true that $X$ has a countable $\pi$-base?
\end{question}

\section{Acknowledgements}

The author is grateful to FAPESP for financial support through postdoctoral grant 2013/14640-1, \emph{Discrete sets and cardinal invariants in set-theoretic topology} and to Lucia Junqueira for helpful discussion.

\end{document}